\documentclass[a4paper,11pt,final]{article}

\usepackage{fullpage}
\usepackage{mathtools}
\usepackage{amssymb}
\usepackage{amsthm}
\usepackage[numbers]{natbib}
\usepackage{doi}
\usepackage{mathrsfs} 
\usepackage{enumitem} 
\usepackage{tikz-cd}
\usepackage{ifthen} 
\usepackage{arydshln} 

\newcounter{letterctr}

\theoremstyle{plain}
\newtheorem{lettered_thm}[letterctr]{Theorem}
\newtheorem{thm}{Theorem}[section]
\newtheorem*{thm*}{Theorem}
\newtheorem{prop}[thm]{Proposition}
\newtheorem{cor}[thm]{Corollary}
\newtheorem{lemma}[thm]{Lemma}

\theoremstyle{definition}
\newtheorem{defn}[thm]{Definition}
\newtheorem{ex}[thm]{Example}

\theoremstyle{remark}
\newtheorem{rmk}[thm]{Remark}
\newtheorem*{rmk*}{Remark}
\newtheorem*{notn}{Notation}

\newcommand{\cat}[1]{\ensuremath{\mathscr{#1}}}
\newcommand{\oper}[1]{\ensuremath{\mathscr{#1}}}
\newcommand{\lmod}[1]{\ensuremath{#1\mbox{--Mod}}}
\newcommand{\FI}{\ensuremath{\mathsf{FI}}}
\newcommand{\SI}{\ensuremath{\mathsf{SI}}}
\newcommand{\Ab}{\ensuremath{\mathsf{Ab}}}
\newcommand{\Grp}{\ensuremath{\mathsf{Grp}}}
\newcommand{\SpZ}[1]{\ensuremath{\mathrm{Sp}(2#1,\mathbb{Z})}}
\newcommand{\IA}{\ensuremath{I\!A}}
\newcommand{\free}{\ensuremath{\mathbb{F}}}
\newcommand{\surf}{\ensuremath{S}}
\newcommand{\sym}{\ensuremath{\Sigma}}
\newcommand{\tor}{\ensuremath{\mathscr{I}}}
\newcommand{\set}[2]{\ensuremath{\left\lbrace#1\mid#2\right\rbrace}}
\newcommand{\Sp}{\ensuremath{\mathrm{Sp}}}
\newcommand{\GL}{\ensuremath{\mathrm{GL}}}
\newcommand{\graph}{\ensuremath{\mathcal{G}}}

\newcommand{\dmo}{\DeclareMathOperator}
\newcommand{\bld}{\mathbf}

\dmo{\End}{End}
\dmo{\Aut}{Aut}
\dmo{\Out}{Out}
\dmo{\Id}{Id}
\dmo{\coker}{coker}
\dmo{\Ker}{Ker}
\dmo{\Coker}{Coker}
\dmo{\id}{id}
\dmo{\im}{im}
\dmo{\graded}{gr}
\dmo{\Hom}{Hom}
\dmo{\Diff}{Diff}
\dmo{\colim}{colim}
\dmo{\Iso}{Iso}

\title{Homological stability for quotients of mapping class groups of surfaces by the Johnson subgroups}
\author{Tom\'{a}\v{s} Zeman}

\begin{document}
\maketitle
\begin{abstract}
We study quotients of mapping class groups \(\Gamma_{g,1}\) of oriented surfaces with one boundary component by terms of their Johnson filtrations, and we show that the homology of these quotients with suitable systems of twisted coefficients stabilises as the genus of the surface goes to infinity. We also compute the stable (co)homology with constant rational coefficients for one family of such quotients.
\end{abstract}

\section{Introduction and statement of results}
Let \(\surf_{g,n}\) be an oriented surface of genus \(g\) with \(n\) boundary components, and let \(\Gamma_{g,n}\) be its mapping class group, i.e. the group of isotopy classes of orientation-preserving homeomorphisms of \(\surf_{g,k}\) that fix a neighbourhood of the boundary pointwise.

Gluing on a pair of pants \(\surf_{0,3}\) along one or two boundary components yields embeddings \(\surf_{g,n}\hookrightarrow\surf_{g,n+1}\) and \(\surf_{g,n}\hookrightarrow\surf_{g+1,n-1}\) respectively, and these in turn give rise to homomorphisms \(\Gamma_{g,n}\to\Gamma_{g,n+1}\) and \(\Gamma_{g,n}\to\Gamma_{g+1,n-1}\) which extend a given mapping class by the identity on \(\surf_{0,3}\). Harer and Ivanov showed that these homomorphisms induce isomorphisms on group homology in degrees which are small compared to the genus. In particular, the sequence \(\surf_{0,1}\hookrightarrow\surf_{1,1}\hookrightarrow\cdots\) of embeddings constructed by attaching \(\surf_{1,2}\) to the unique boundary component of \(\surf_{\ast,1}\) yields a sequence of homomorphisms \(\Gamma_{0,1}\to\Gamma_{1,1}\to\ldots\) which satisfies homological stability in the sense that
\[H_i(\Gamma_{g-1,1};\mathbb{Z})\to H_i(\Gamma_{g,1};\mathbb{Z})\]
is an isomorphism whenever \(i\leqslant\frac{2}{3}g-\frac{4}{3}\) (see \emph{e.g.} Wahl's survey \cite{wah14}).

If we fix a base-point \(\ast\in\partial\surf_{g,1}\), the fundamental group \(\pi\coloneqq\pi_1(\surf_{g,1},\ast)\) enjoys an action of \(\Gamma_{g,1}\). This action preserves the lower central series\footnote[2]{We use the indexing convention \(\gamma_1(\pi)=\pi\).} \(\{\gamma_k(\pi)\}_{k\geqslant1}\) of \(\pi\), and the \emph{Johnson filtration} \(\{\tor_{g,1}(k)\}_{k\geqslant1}\) of \(\Gamma_{g,1}\) consists of kernels of the induced actions of \(\Gamma_{g,1}\) on the quotients \(\pi/\gamma_{k+1}(\pi)\). It is easy to show that the above homomorphisms \(\Gamma_{g-1,1}\to\Gamma_{g,1}\) preserve the Johnson filtration. Hence a natural question arises whether the induced homomorphisms
\(\Gamma_{g-1,1}/\tor_{g-1,1}(k)\to\Gamma_{g,1}/\tor_{g,1}(k)\) are also homology isomorphisms when \(g\) is sufficiently large.

\begin{lettered_thm}\label{thm:hom stab}
Let \(k\geqslant1\), and let \(M_0\to M_1\to\cdots\) be a good module of the sequence of groups \(\Gamma_{0,1}/\tor_{0,1}(k)\to\Gamma_{1,1}\tor_{1,1}(k)\to\cdots\). Then for every \(i\), the homomorphisms in the sequence
\[\cdots\to H_i(\Gamma_{g-1,1}/\tor_{g-1,1}(k);M_{g-1})\to H_i(\Gamma_{g,1}/\tor_{g,1}(k);M_g)\to\cdots\]
are eventually (for large \(g\)) isomorphisms.
\end{lettered_thm}

See Definition~\ref{def:good_module} for the precise meaning of a good module of a sequence of groups; in essence, these are modules that come from well-behaved sequences of representations of the symplectic groups \(\SpZ{g}\). In particular, for any abelian group \(A\), the constant coefficients in \(A\) constitute a good module.

Our proof is by induction, bootstrapping known results about homological stability of the symplectic groups \(\SpZ{g}\cong\Gamma_{g,1}/\tor_{g,1}(1)\) to the higher quotients \(\Gamma_{g,1}/\tor_{g,1}(k)\), \(k>1\), by means of a spectral sequence argument applied to the extensions
\[0\to\tor_{g,1}(k-1)/\tor_{g,1}(k)\to\Gamma_{g,1}/\tor_{g,1}(k)\to\Gamma_{g,1}/\tor_{g,1}(k-1)\to1.\]
Our approach has a drawback in that we get no hold on the stable range: the heart of our argument is Theorem~\ref{Qk_finite_degree} which says that certain \(\FI\)-modules which encapsulate the successive quotients \(\tor_{g,1}(k-1)/\tor_{g,1}(k)\) are finitely generated. Its proof relies on the noetherian property of the category of \(\FI\)-modules over \(\mathbb{Z}\) which was proved by Church, Ellenberg, Farb and Nagpal in \cite{cefn14}. This property is non-constructive in the sense that it guarantees the existence of a finite generating set of a sub-\(\FI\)-module of a finitely generated \(\FI\)-module, but it offers no information on the size or degrees of some finite set of generators.

The inspiration to use the above bootstrapping method came from a paper of Szymik \cite{szy14} where this method is applied to the closely related automorphism groups of free nilpotent groups. Szymik considers the extensions
\[0\to K_n(k)\to\Aut N_n(k)\to\Aut N_n(k-1)\to1\]
where \(N_n(k)\) is the free nilpotent group on \(n\) generators, of nilpotency class \(k\). If we write \(\free_n\) for the free group on \(n\) generators, then \(N_n(k)\cong\free_n/\gamma_{k+1}(\free_n)\). The kernels \(K_n(k)\) in the above extensions are relatively easy to understand, allowing for extraction of linear stability bounds from Szymik's argument. Indeed, one can show that
\[K_n(k)\cong\Hom(\mathbb{Z}^n,\gamma_k(\free_n)/\gamma_{k+1}(\free_n))\]
which is a free abelian group. The quotients \(\tor_{g,1}(k-1)/\tor_{g,1}(k)\), on the other hand, seem more difficult to analyse, and while they become subgroups of \(K_{2g}(k)\) upon fixing an isomorphism \(\pi_1(\surf_{g,1},\ast)\cong\free_{2g}\), it is not clear to us \emph{which} subgroups. It is for this reason that we use the techniques of representation stability and consequently lose information about the stable range.

The pair-of-pants product endows the space \(\coprod_{g\geqslant0}B\Gamma_{g,1}\) with a monoid structure, and the group completion of this monoid is (weakly equivalent to) \(\mathbb{Z}\times B\Gamma_{\infty}^+\) by the group-completion theorem, where \(\Gamma_{\infty}=\colim_g\Gamma_{g,1}\) and \(X^+\) is the Quillen plus construction on \(X\) with respect to the maximal perfect subgroup of \(\pi_1(X)\). Tillmann \cite{til97,til00} showed that \(\mathbb{Z}\times B\Gamma_{\infty}^+\) is in fact an infinite loop space, and that the obvious map \(\mathbb{Z}\times B\Gamma_{\infty}^+\to K_{\Sp}(\mathbb{Z})\) to the Hermitian \(K\)-theory of \(\mathbb{Z}\) is a map of infinite loop spaces. We show that analogous constructions for the quotients \(\Gamma_{g,1}/\tor_{g,1}(k)\) also yield infinite loop spaces and that we get an infinite tower of infinite loop spaces interpolating between \(\mathbb{Z}\times B\Gamma_{\infty}^+\) and \(K_{\Sp}(\mathbb{Z})\).

\begin{lettered_thm}\label{thm:inf loops}
The obvious quotient homomorphisms of groups give rise to a commutative diagram
\[
\begin{tikzcd}
\mathbb{Z}\times B\Gamma_{\infty}^+ \ar[drr] \ar[drrr] \ar[drrrr] &[-7em] & & & &[-2em]\\
& \cdots \ar[r] & \mathbb{Z}\times B\Gamma_{\infty}(3)^+ \ar[r] & \mathbb{Z}\times B\Gamma_{\infty}(2)^+ \ar[r] & \mathbb{Z}\times B\Gamma_{\infty}(1)^+ \ar[r, phantom, "\simeq" description] & K_{\Sp}(\mathbb{Z})
\end{tikzcd}
\]
of maps of infinite loop spaces.
\end{lettered_thm}

\paragraph{Outline of paper}
In Section~\ref{section:mods of fin deg}, we discuss \(\cat{C}\)-modules, \emph{i.e.} functors from a small category \(\cat{C}\) to the category of abelian groups, and we introduce the two choices of \(\cat{C}\) that will be of interest to us, namely the category \(\FI\) of finite sets and injections, and the category \(\SI\) of finitely generated free symplectic spaces over \(\mathbb{Z}\) and symplectic maps. In Sections~\ref{section:automs of free nilps} and \ref{section:quotients of mcgs}, we recall some facts about automorphism groups of free nilpotent groups, we construct the \(\SI\)-modules \(Q(k)\) which encapsulate the successive quotients \(\tor_{g,1}(k-1)/\tor_{g,1}(k)\) as \(\SpZ{g}\)-modules, and we prove the key Theorem~\ref{Qk_finite_degree}. Sections~\ref{section:proof of theorem} and \ref{section:infinite loop space} are devoted to the proofs of Theorem~\ref{thm:hom stab} and \ref{thm:inf loops}, respectively.

\subsection{Acknowledgements}
I would like to thank my supervisor Ulrike Tillmann for her constant encouragement and many useful discussions, as well as for her suggestions which greatly improved the exposition of this paper.

\subsection{Notation}
We use the following conventions: \(\mathbb{N}\coloneqq\lbrace0,1,2,\ldots\rbrace\) is the set of natural numbers including zero; \(\Grp\) and \(\Ab\) are the categories of groups and abelian groups, respectively; \(\sym_n\) is the symmetric group on \(n\) letters; \(\Sp(2n,R)\) is the group of \(2n\times2n\) symplectic matrices with entries in a ring \(R\) --- for us, \(R=\mathbb{Z}\) most of the time. In a group \(G\), we write \([x,y]=xyx^{-1}y^{-1}\) for the commutator of \(x,y\in G\).

\section{Modules of finite degree}\label{section:mods of fin deg}
\subsection{Modules over small categories}
Let \(\cat{C}\) be a small category with objects (canonically identified with) the natural numbers \(\mathbb{N}\).

\begin{defn}
A \(\cat{C}\)-\emph{object in a category \(\cat{A}\)} is a functor \(\cat{C}\to\cat{A}\). A morphism of \(\cat{C}\)-objects in \(\cat{A}\) is a natural transformation of functors.
\end{defn}

\begin{notn}
We write \(F_n\) for the value of a \(\cat{C}\)-object \(F\) at \(n\in\cat{C}\). When \(f:n\to m\) is an arrow in \(\cat{C}\), we write \(f_{\ast}:F_n\to F_m\) for the induced morphism in \(\cat{A}\). When \(\phi:F\to F'\) is a morphism of \(\cat{C}\)-objects, we write \(\phi_n:F_n\to F'_n\) for its component at \(n\in\cat{C}\).
\end{notn}

\begin{defn}
A \(\cat{C}\)-\emph{module} is a \(\cat{C}\)-object in \(\Ab\). We write \(\lmod{\cat{C}}\) for the category of \(\cat{C}\)-modules and natural transformations.
\end{defn}

\begin{rmk}
\(\lmod{\cat{C}}\) is abelian, with kernels and cokernels computed pointwise. Hence it makes good sense to talk of sub-\(\cat{C}\)-modules, quotient \(\cat{C}\)-modules \emph{etc.}
\end{rmk}

Now suppose further that the category \(\cat{C}\) is strict monoidal, where the monoidal product \(\oplus\) is given by \(n\oplus m=n+m\) on objects, and that the monoidal unit \(0\) is initial in \(\cat{C}\).

\begin{defn}
We will call such \(\cat{C}\) a \emph{category with objects the naturals}, or CON for short.
\end{defn}

Consider the functor \(R\coloneqq1\oplus(-):\cat{C}\to\cat{C}\). Since \(0\) is initial, the unique arrow \(0\to 1\) induces a natural transformation \(\rho:\Id\Rightarrow R\) from the identity \(\Id=0\oplus(-)\) to \(R\).

\begin{defn}\label{defn:ker and coker}
Let \(F\) be a \(\cat{C}\)-module. The \emph{suspension} of \(F\) is \(\Sigma F\coloneqq F\circ R\). We define two new \(\cat{C}\)-modules, called the \emph{kernel} and \emph{cokernel} of \(F\) and denoted \(\Ker F\) and \(\Coker F\), to be the kernel and cokernel of the morphism \(F\rho:F\to\Sigma F\) of \(\cat{C}\)-modules.
\end{defn}

\begin{notn}
To avoid any confusion, \(\ker\) and \(\coker\) with small initial letters will always refer to the usual kernel and cokernel in \(\lmod{\cat{C}}\) (or elsewhere) as an abelian category, whereas \(\Ker\) and \(\Coker\) with capitals will always denote the above constructions.
\end{notn}

\begin{ex}
When \(\cat{C}\) is the poset \(\mathbb{N}\) with the evident monoidal structure, a \(\cat{C}\)-module \(F\) is just a composable sequence \(F_0\to F_1\to\cdots\) of homomorphisms of abelian groups. The morphism \(F\rho:F\to\Sigma F\) is
\begin{equation*}
\begin{tikzcd}
F_0 \ar[r] \ar[d] & F_1 \ar[r] \ar[d] & \cdots\\
F_1 \ar[r] & F_2 \ar[r] & \cdots
\end{tikzcd}
\end{equation*}
\(\Ker F\) consists of the kernels of the homomorphisms in \(F\) and zero maps, and similarly for the cokernel.
\end{ex}

Observe that \(\Ker\) and \(\Coker\) are in fact functors \(\lmod{\cat{C}}\to\lmod{\cat{C}}\). We have the following easy lemma.

\begin{lemma}\label{l:ex seq of ker and coker}
If \(0\to F'\xrightarrow{\iota}F\xrightarrow{\pi}F''\to0\) is a short exact sequence of \(\cat{C}\)-modules, there is an exact sequence
\[0\to\Ker F'\xrightarrow{\Ker\iota}\Ker F\xrightarrow{\Ker\pi}\Ker F''\to\Coker F'\xrightarrow{\Coker\iota}\Coker F\xrightarrow{\Coker\pi}\Coker F''\to0.\]
\end{lemma}
\begin{proof}
Immediate from the snake lemma.
\end{proof}

Let \(\cat{D}\) be another CON, and let \(X:\cat{D}\to\cat{C}\) be a strict monoidal functor which is the identity on objects. \(X\) induces a pullback functor \(X^{\ast}:\lmod{\cat{C}}\to\lmod{\cat{D}}\) taking \(F\) to \(F\circ X\). Moreover, we have the following lemma.
\begin{lemma}\label{l:pullbacks commute with everything}
\(X^{\ast}\) commutes with the functors \(\Sigma\), \(\Ker\) and \(\Coker\). In other words, for a \(\cat{C}\)-module \(F\), \(X^{\ast}(\Sigma F)=\Sigma(X^{\ast}F)\), \emph{etc}.
\end{lemma}
\begin{proof}
This is immediate once we note that since \(X\) is strict monoidal and the identity on objects, we have \(R_{\cat{C}}X=X R_{\cat{D}}\) and \(\rho_{\cat{C}}X=X\rho_{\cat{D}}\). Here \(R_{\cat{C}}\) and \(\rho_{\cat{C}}\) denote the above functor and natural transformation in \(\cat{C}\), and similarly for \(\cat{D}\).
\end{proof}

\begin{defn}
We will call a strict monoidal functor \(X:\cat{D}\to\cat{C}\) which is the identity on objects a \emph{functor of categories with objects the naturals}, or a \emph{functor of CONs} for short.
\end{defn}

\subsection{The degree of a \(\cat{C}\)-module}
Let \(\cat{C}\) be a CON, and let \(F\) be a \(\cat{C}\)-module.
\begin{defn}
\(F\) has degree \(-1\) if \(F_n=0\) for all sufficiently large \(n\). For \(d\geqslant0\), \(F\) has degree \(d\) if \(\Ker F\) has degree \(-1\) and \(\Coker F\) has degree \(d-1\).
\(F\) has finite degree if \(F\) has degree \(d\) for some \(d\geqslant -1\).
\end{defn}

The key observation of this section is the following lemma.
\begin{lemma}\label{l:degree of pullback}
Suppose \(X:\cat{D}\to\cat{C}\) is a functor of CONs. Then a \(\cat{C}\)-module \(F\) has degree \(d\) if and only if the \(\cat{D}\)-module \(X^{\ast}F\) has degree \(d\).
\end{lemma}
\begin{proof}
This follows from an easy induction on \(d\) using Lemma~\ref{l:pullbacks commute with everything}. Since \(X\) is the identity on objects, \((X^{\ast}F)_n=F_n\), so \(F\) has degree \(-1\) if and only if \(X^{\ast}F\) does. Assuming the result for \(d-1\), the following are equivalent:
\begin{itemize}
\item \(F\) has degree \(d\);
\item \(\Ker F\) has degree \(-1\) and \(\Coker F\) has degree \(d-1\);
\item \(X^{\ast}\Ker F=\Ker X^{\ast}F\) has degree \(-1\) and \(X^{\ast}\Coker F=\Coker X^{\ast}F\) has degree \(d-1\);
\item \(X^{\ast}F\) has degree \(d\).
\end{itemize}
\end{proof}

\subsection{The categories \(\FI\) and \(\SI\)}
In this section we introduce two more examples of CONs beyond \(\mathbb{N}\), namely the categories \(\FI\) of finite sets and injections, and \(\SI\) of free finitely generated abelian groups with symplectic forms and symplectic injections.
\begin{defn}
Let \(\FI\) be the category of sets of the form \(\bld{n}\coloneqq\lbrace1,2,\ldots,n\rbrace\), \(n\geqslant0\), and set injections.
\end{defn}

\(\FI\) is strict monoidal with the monoidal product \(\sqcup\) coming from disjoint union of sets: \(\bld{n}\sqcup\bld{m}=\bld{n+m}=\lbrace1,2,\ldots,n+m\rbrace\) on objects and
\[(f\sqcup g)(i)=
\begin{cases}
f(i) & \mbox{if }i\leqslant n\\
g(i-n)+m & \mbox{if }i>n
\end{cases}
\]
on morphisms \(f:\bld{n}\to\bld{m}\) and \(g:\bld{k}\to\bld{l}\). It is easily seen that this is indeed a CON.

\begin{defn}
Let \(V_n\), \(n\geqslant0\), be the free abelian group with basis \(a_1,b_1,a_2,b_2,\ldots,a_n,b_n\), equipped with the symplectic form \((\cdot,\cdot)\) given by
\begin{gather*}
(a_i,a_j)=0=(b_i,b_j)\\
(a_i,b_j)=\delta_{ij}=-(b_j,a_i)
\end{gather*}
for all \(1\leqslant i,j\leqslant n\), where \(\delta_{ij}\) is the Kronecker delta. Let \(\SI\) be the category with objects the \(V_n\) and arrows the form-preserving homomorphisms.
\end{defn}

The following lemma is standard linear algebra. We quote it here for later reference.
\begin{lemma}\label{l:SI facts}\leavevmode
\begin{enumerate}[label={\normalfont (\roman*)}]
\item \(\End_{\SI}(V_n)=\Aut_{\SI}(V_n)=\SpZ{n}\)
\item For every \(f:V_n\to V_m\) in \(\SI\), there is a unique \(C\leqslant V_m\) such that \(V_m\cong(\operatorname{im}f)\oplus C\) as abelian groups with bilinear forms (where the form on \(C\) is the restriction of the form on \(V_m\)). Moreover, \(C\cong V_{m-n}\).
\item If \(f\) and \(C\) are as above and \(s\in\Aut_{\SI}(V_m)\) satisfies \(f=s\circ f\), then \(s\) preserves the above direct sum decomposition and \(s=\id\oplus s'\) for some \(s'\in\Aut(C)\).
\end{enumerate}
\end{lemma}

\(\SI\) is strict monoidal with monoidal product \(\boxtimes\) coming from direct sum. Writing \(V_n\oplus V_m\) for the usual direct sum of abelian groups with bilinear forms, let \(\alpha_{n,m}:V_n\oplus V_m\xrightarrow{\sim}V_{n+m}\) be the obvious form-preserving isomorphism which takes the bases of \(V_n\) and \(V_m\) to the first \(2n\) and last \(2m\) basis elements of \(V_{n+m}\), respectively, in an order-preserving fashion. We define \(V_n\boxtimes V_m\coloneqq V_{n+m}\) and we let \(f\boxtimes g\) be the composition
\[V_{n+k}\xrightarrow{\alpha_{n,k}^{-1}}V_n\oplus V_k\xrightarrow{f\oplus g}V_m\oplus V_l\xrightarrow{\alpha_{m,l}}V_{m+l}\]
for morphisms \(f:V_n\to V_m\), \(g:V_k\to V_l\). Again it is easily checked that this makes \(\SI\) into a CON.

\begin{defn}\label{defn:SI-Mod to FI-Mod}
Let \(X:\FI\to\SI\) be the functor which sends \(\bld{n}\) to \(V_n\) and \(f:\bld{n}\to\bld{m}\) to the homomorphism \(Xf:V_n\to V_m\) defined on the standard basis by \(Xf(a_i)=a_{f(i)}\) and \(Xf(b_i)=b_{f(i)}\).
\end{defn}
It is readily seen that \(X\) is a functor of CONs, and so Lemma~\ref{l:degree of pullback} applies.

\begin{rmk*}
The main reasons why the definition of \(V_n\) contains a reference to a specific basis is that it gives us for free a coherent choice of the isomorphisms \(\alpha_{n,m}\), and the construction of the above functor \(X\) becomes rather easy. If we defined the \(V_n\) simply as free abelian groups of rank \(2n\) equipped with a symplectic form, we would have to do more work to construct \(\boxtimes\) and \(X\).
\end{rmk*}

\subsection{\(\FI\)-modules}\label{subsection:FI modules}
In this section, we investigate \(\FI\)-modules in greater detail, and give a complete characterisation of \(\FI\)-modules of finite degree.

\paragraph{Basic notions}
\(\FI\)-modules have been much studied in recent years in connection to stable representation theory of symmetric groups, and there is a wealth of literature on them. We define all the relevant terms and state the results we will be using later on, but see e.g. \cite{cefn14} for a more detailed introduction.
\begin{defn}
For \(n\geqslant0\), the \emph{\(n\)-th principal projective} \(P^n\) is the \(\FI\)-module which sends \(\bld{m}\) to the free abelian group generated by the set \(\FI(\bld{n},\bld{m})\), and similarly for arrows.

An \(\FI\)-module is \emph{free} if it is a direct sum of principal projectives.
\end{defn}

The \(P^n\) are indeed projective objects of \(\lmod{\FI}\). By the Yoneda lemma, there is an isomorphism \(\lmod{\FI}(P^n,F)\cong F_n\), natural in \(F\in\lmod{\FI}\), which sends \(\phi:P^n\to F\) to the image under \(\phi_n:P_n^n\to F_n\) of the identity \(\id_{\bld{n}}\in P_n^n\). The image of \(\phi\) in \(F\) can be seen to be the least sub-\(\FI\)-module of \(F\) which contains \(\phi_n(\id_{\bld{n}})\in F_n\).

Similarly, homomorphisms
\[\bigoplus_{i\in I}P^{n_i}\to F\]
from a fixed free \(\FI\)-module biject naturally with \(I\)-indexed sequences \((x_i\in F_{n_i})_{i\in I}\), and the image of such a homomorphism is the least sub-\(\FI\)-module of \(F\) containing \(x_i\in F_{n_i}\) for all \(i\) (\cite{cefn14}). This motivates the following definition.

\begin{defn}
Let \(F\) be an \(\FI\)-module. We say \(F\) is \emph{generated in degrees \(\leqslant k\)} if \(F\) receives a surjection
\begin{equation}\label{defn:gend in finite degrees}
\bigoplus_{i\in I}P^{n_i}\twoheadrightarrow F
\end{equation}
from a free module with \(n_i\leqslant k\) for all \(i\in I\). \(F\) is \emph{generated in finite degree} if it is generated in degrees \(\leqslant k\) for some \(k\).

\(F\) is \emph{finitely generated} if there is a surjection \eqref{defn:gend in finite degrees} with \(I\) finite.

We say \(F\) is \emph{generated in degrees \(\leqslant k\) and related in degrees \(\leqslant d\)} if there is an exact sequence
\begin{equation}\label{defn:gend and reld in finite degrees}
\bigoplus_{j\in J}P^{m_j}\to\bigoplus_{i\in I}P^{n_i}\to F\to0
\end{equation}
where each \(n_i\leqslant k\) and each \(m_j\leqslant d\). \(F\) is \emph{generated and related in finite degrees} if it is generated in degrees \(\leqslant k\) and related in degrees \(\leqslant d\) for some \(k\) and \(d\).
\end{defn}

\begin{rmk}\label{alt_defn:gend and reld in finite degrees}
Note that a surjection \(\bigoplus_{i\in I}P^{n_i}\twoheadrightarrow F\) extends to an exact sequence as in \eqref{defn:gend and reld in finite degrees} if and only if its kernel is generated in degrees \(\leqslant d\). Hence an equivalent definition of \(F\)  being generated in degrees \(\leqslant k\) and related in degrees \(\leqslant d\) is that there is a short exact sequence
\[0\to K\to\bigoplus_{i\in I}P^{n_i}\to F\to0\]
where each \(n_i\leqslant k\) and where \(K\) is generated in degrees \(\leqslant d\).
\end{rmk}

The main result of \cite{cefn14} (and the reason why \(\FI\)-modules are useful to us) is the following theorem.

\begin{thm}[\cite{cefn14}]\label{thm:FI noetherian}
The category \(\lmod{\FI}\) is noetherian: every sub-\(\FI\)-module of a finitely generated \(\FI\)-module is finitely generated.
\end{thm}

\paragraph{\(\FI\)-modules of finite degree}
We record the following result from \cite[Section 4]{chel16} which will be useful later on.

\begin{lemma}\label{l:coker and finite degs}
An \(\FI\)-module \(F\) is generated in finite degree if and only if \(\Coker^n F\) has degree \(-1\) for some \(n\) (where \(\Coker\) is the functor from Definition~\ref{defn:ker and coker}).
\end{lemma}
\begin{proof}
\cite{chel16} proves the following closely related statement: \(F\) is generated in finite degree if and only if \(\Coker^n F=0\) for some \(n\). Our statement is equivalent: the zero \(\FI\)-module certainly has degree \(-1\). Conversely, if \(\Coker^nF\) has degree \(-1\), then there exists \(M\) such that \((\Coker^nF)_m=0\) for all \(m\geqslant M\). But then \(\Coker^{n+M}F=0\).
\end{proof}

It can be seen that the definition of an \(\FI\)-module \(F\) having finite degree does not depend on the structure of \(F\) in any finite set of degrees: \(F\) has degree \(-1\) if it is \emph{eventually} 0, and higher degrees are defined recursively. Hence the following definition will be useful in our investigation of modules of finite degree.

\begin{defn}
We say that a morphism \(\phi:F\to F'\) of \(\FI\)-modules is \emph{eventually isomorphic} if \(\phi_n:F_n\to F'_n\) is an isomorphism for all sufficiently large \(n\).
\end{defn}

Suppose that \(0\to K\to P\xrightarrow{\pi} F\to0\) is a short exact sequence of \(\FI\)-modules. \(\Coker\) is right-exact, so for every \(i\geqslant0\), the sequence
\[\Coker^iK\to\Coker^iP\xrightarrow{\Coker^i\pi}\Coker^iF\to0\]
is exact. Let \(K^{(i)}\) be the kernel of the morphism \(\Coker^i\pi\), so \(K^{(0)}=K\) and there is an induced morphism \(u_i:\Coker^iK\to K^{(i)}\).

\begin{lemma}\label{cokers_event_isom}
Let \(0\to K\to P\to F\to0\) and \(K^{(i)}\) be as above, and suppose that \(F\) has finite degree. Then for every \(i\), \(u_i:\Coker^iK\to K^{(i)}\) is eventually isomorphic.
\end{lemma}
\begin{proof}
We proceed by induction on \(i\). When \(i=0\), there is nothing to prove.

Suppose the claim has been established for \(i-1\). Applying \(\Coker\) to the short exact sequence \(0\to K^{(i-1)}\to\Coker^{i-1}P\to\Coker^{i-1}F\to0\) yields the exact sequence
\begin{equation*}
\begin{tikzcd}
\Ker\Coker^{i-1}F \ar[r] & \Coker K^{(i-1)} \ar[rr] \ar[rd, two heads, "q"'] && \Coker^iP \ar[r] & \Coker^iF \ar[r] & 0\\
&& K^{(i)} \ar[ru, hook]
\end{tikzcd}
\end{equation*}
by Lemma~\ref{l:ex seq of ker and coker}. Since \(F\) has finite degree, \(\Ker\Coker^{i-1}F\) has degree \(-1\), so the morphism \(q:\Coker K^{(i-1)}\to K^{(i)}\) is eventually isomorphic. \(\Coker u_{i-1}\) is eventually isomorphic by the inductive hypothesis. Hence \(u_i\), which is equal to the composition
\[\Coker^iK\xrightarrow{\Coker u_{i-1}}\Coker K^{(i-1)}\xrightarrow{q} K^{(i)},\]
is also eventually isomorphic.
\end{proof}

Now we are ready to characterise \(\FI\)-modules of finite degree.

\begin{prop}\label{fin_deg_character}
Let \(F\) be an \(\FI\)-module. Then \(F\) has finite degree if and only if \(F\) is generated and related in finite degrees.
\end{prop}
\begin{proof}
The statement that \(\FI\)-modules which are generated and related in finite degrees have finite degree, is precisely the content of \cite[Proposition~4.18]{rww15}.

To prove the other direction, suppose \(F\) has finite degree. Then \(F\) is generated in finite degree by Lemma~\ref{l:coker and finite degs}, so there is a short exact sequence
\[0\to K\to L\to F\to0\]
where \(L\) is free and generated in finite degree. If we can show that \(K\) is generated in finite degree, we are done by Remark~\ref{alt_defn:gend and reld in finite degrees}.

Using the notation of Lemma~\ref{cokers_event_isom}, there are eventual isomorphisms \(\Coker^i K\to K^{(i)}\). But by Lemma~\ref{l:coker and finite degs}, \(\Coker^iL\) has degree \(-1\) for some \(i\). Since \(K^{(i)}\) is a sub-\(\FI\)-module of \(\Coker^iL\), it also has degree \(-1\), and so \(\Coker^iK\) has degree \(-1\) by virtue of being eventually isomorphic to \(K^{(i)}\). Hence \(K\) is generated in finite degree.
\end{proof}

\begin{cor}
Finitely generated \(\FI\)-modules have finite degree.
\end{cor}
\begin{proof}
Suppose \(F\) is a finitely generated \(\FI\)-module, \emph{i.e.} there is a surjection as in \eqref{defn:gend in finite degrees} with \(I\) finite. The kernel of this surjection is finitely generated by Theorem~\ref{thm:FI noetherian}, so \(F\) is generated and related in finite degrees.
\end{proof}

To prove our second corollary, we need the following technical lemma.

\begin{lemma}\label{P_tensors}
For every \(n,m\), the \(\FI\)-module \(P^n\otimes P^m:\FI\xrightarrow{P^n\times P^m}\Ab\times\Ab\xrightarrow{\otimes}\Ab\) is free and finitely generated in degrees \(\leqslant n+m\).
\end{lemma}
\begin{proof}
Write \(F\coloneqq P^n\otimes P^m\). For every \(k\), \(F_k=P_k^n\otimes P_k^m\) is a free abelian group with basis all pairs \((f:\bld{n}\hookrightarrow\bld{k},g:\bld{m}\hookrightarrow\bld{k})\) of injections. To each such pair, we associate the following data:
\begin{align*}
N&\coloneqq f^{-1}(\im f\cap\im g)\subset\bld{n}\\
M&\coloneqq g^{-1}(\im f\cap\im g)\subset\bld{m}\\
\sigma&\coloneqq N\xrightarrow{f}\im f\cap\im g\xrightarrow{g^{-1}}M.
\end{align*}
Now observe that if \(h:\bld{k}\to\bld{l}\) is a morphism in \(\FI\), then \(h_{\ast}(f,g)=(h\circ f,h\circ g)\in F_l\) has the same associated data as \((f,g)\). It follows that \(F\) splits as a direct sum of factors \(T^{(N,M,\sigma)}\) indexed by triples \((N,M,\sigma)\) where \(N\subset\bld{n}\), \(M\subset\bld{m}\) and \(\sigma:N\xrightarrow{\sim}M\) is an isomorphism: \(T_k^{(N,M,\sigma)}\) is the subgroup of \(F_k\) generated by precisely those basis elements \((f,g)\) whose associated data are the prescribed \((N,M,\sigma)\).

Now we claim that each \(T^{(N,M,\sigma)}\) is in fact a principal projective. Indeed, generators \((f,g)\) of \(T_k^{(N,M,\sigma)}\) biject naturally (in \(\bld{k}\in\FI\)) with maps \(\bld{n}\cup_{\sigma}\bld{m}\to\bld{k}\), where \(\bld{n}\cup_{\sigma}\bld{m}\) is the pushout of \(\bld{n}\hookleftarrow N\xrightarrow{\sigma}\bld{m}\). Hence \(T^{(N,M,\sigma)}\) is isomorphic to the principal projective generated in degree \(\vert\bld{n}\cup_{\sigma}\bld{m}\vert=n+m-\vert N\vert\leqslant n+m\).
\end{proof}

\begin{rmk*}
Note that \(\operatorname{rank}_{\mathbb{Z}}P_k^n=n!\binom{k}{n}\). By computing the rank of \(P_k^n\otimes P_k^m\) in two ways using the above proof, we obtain the following rather nice identity
\[\sum_{a}a!(n+m-a)!\binom{n}{a}\binom{m}{a}\binom{k}{n+m-a}=n!m!\binom{k}{n}\binom{k}{m}\]
where \(a!\binom{n}{a}\binom{m}{a}\) is the number of different triples \((N,M,\sigma)\) with \(\vert N\vert=a\).
\end{rmk*}

\begin{cor}\label{fin_deg_tensors}
If \(F,F'\) are \(\FI\)-modules of finite degree, then so is the degree-wise tensor product \(F\otimes F':\FI\xrightarrow{F\times F'}\Ab\times\Ab\xrightarrow{\otimes}\Ab\).
\end{cor}
\begin{proof}
By Proposition~\ref{fin_deg_character}, there are exact sequences
\begin{gather*}
L_1\to L_0\to F\to0\\
L'_1\to L'_0\to F'\to0
\end{gather*}
where the \(L_i,L'_i\) are free and generated in finite degree. Considering the total complex of the bicomplex \(L_{\ast}\otimes L'_{\ast}\) yields the exact sequence
\[(L_0\otimes L'_1)\oplus(L_1\otimes L'_0)\to L_0\otimes L'_0\to F\otimes F'\to 0.\]
But the first two terms in the sequence are free and generated in finite degree by the previous lemma, and so \(F\otimes F'\) has finite degree.
\end{proof}

\begin{cor}[of Lemma~\ref{P_tensors}; see also \cite{cefn14}]\label{tensor_products_fg}
If \(F,F'\) are finitely generated \(\FI\)-modules, then so is \(F\otimes F'\).
\end{cor}

We end this section with the following lemma. Using the monoidal structure of \(\FI\), we can define the `duplication functor'
\[D:\FI\xrightarrow{\Delta}\FI\times\FI\xrightarrow{\sqcup}\FI\]
where \(\Delta\) is the usual diagonal functor. \(D\) is clearly not a functor of CONs, so Lemma~\ref{l:degree of pullback} does not apply, but we still have the following result.

\begin{lemma}\label{duplication_fg}
\(D^{\ast}:\lmod{\FI}\to\lmod{\FI}\) takes finitely generated \(\FI\)-modules to finitely generated \(\FI\)-modules.
\end{lemma}
\begin{proof}
\(D^{\ast}\) clearly preserves epimorphisms (in fact it is exact). Hence it suffices to prove that \(D^{\ast}P^n\) is finitely generated for every \(n\). We show \(D^{\ast}P^n\) is generated in degrees \(\leqslant n\). Since for every \(m\), \((D^{\ast}P^n)_m=P_{2m}^n\) is a finitely generated abelian group, this will complete the proof.

Let \(f:\bld{n}\to\bld{2m}\) be a generator of \((D^{\ast}P^n)_m\) where \(m>n\). Then there is some \(i\in\bld{m}\) such that \emph{both} \(i,i+m\in\bld{2m}\) are \emph{not} in the image of \(f\). Let \(g:\bld{m-1}\to\bld{m}\) be the unique order-preserving injection which misses \(i\). Define \(f':\bld{n}\to\bld{2m-2}\) by
\[
f'(j)=
\begin{cases}
f(j) & \mbox{if } f(j)<i,\\
f(j)-1 & \mbox{if } i<f(j)<i+m,\\
f(j)-2 & \mbox{if } i+m<f(j).
\end{cases}
\]
Then \(f=(g\sqcup g)\circ f'\), so \(g_{\ast}:(D^{\ast}P^n)_{m-1}\to(D^{\ast}P^n)_m\) takes \(f'\) to \(f\).
\end{proof}

\section{Automorphisms of free nilpotent groups}\label{section:automs of free nilps}
\paragraph{Lower central series}
Recall that the lower central series \(\lbrace\gamma_k(G)\rbrace_{k\geqslant1}\) of a group \(G\) is defined inductively by setting \(\gamma_1(G)=G\) and \(\gamma_(G)=[G,\gamma_{k-1}(G)]\) for \(k>1\). The successive quotients \(\graded_kG\coloneqq\gamma_k(G)/\gamma_{k+1}(G)\) are abelian. In fact, \([\gamma_k(G),\gamma_l(G)]\leqslant\gamma_{k+l}(G)\), and so the commutator bracket descends to a Lie bracket on the associated graded group \(\graded G=\bigoplus_{k\geqslant1}\graded_kG\). These constructions give rise to functors \(\gamma_k:\Grp\to\Grp\) and \(\graded:\Grp\to\mathsf{grLie}_{\mathbb{Z}}\), where \(\mathsf{grLie}_{\mathbb{Z}}\) is the category of graded Lie algebras over \(\mathbb{Z}\).

Hence for an \(\FI\)-group\footnote[2]{\emph{I.e.} an \(\FI\)-object in the category \(\Grp\).} \(G:\FI\to\mathsf{Grp}\), we obtain a filtration \(\gamma_k(G)\) of \(G\) by sub-\(\FI\)-groups, and a graded \(\FI\)-Lie algebra \(\graded G\) consisting of \(\FI\)-modules \(\graded_kG\).

\begin{defn}
Let \(\free:\FI\to\Grp\) be the restriction to \(\FI\) of the free group functor \(\mathsf{Set}\to\Grp\).
\end{defn}
Thus \(\free_n\) is the free group generated by the set \(\bld{n}\), although we will denote its generators by \(x_1,\ldots,x_n\) to avoid confusion.

\begin{lemma}\label{graded_pieces}
Each \(\graded_k\free\) is a finitely generated \(\FI\)-module taking values in free abelian groups.
\end{lemma}
\begin{proof}
The fact that \((\graded_k\free)_n=\gamma_k(\free_n)/\gamma_{k+1}(\free_n)\) is a finitely generated free abelian group is standard --- see e.g. \cite{ser06}.

We prove by induction on \(k\) that \(\graded_k\free\) is generated in degrees \(\leqslant k\). The case \(k=1\) is clear: \(\graded_1\free\cong P^1\) is the abelianisation of \(\free\), and therefore generated by the image of the generator \(x_1\in \free_1\cong\mathbb{Z}\).

Assume the claim has been established for all \(k<k_0\). The graded Lie algebra \(\graded \free_n\) is generated by \(\graded_1\free_n\), so \(\graded_{k_0}\free_n\) is spanned over \(\mathbb{Z}\) by elements of the form \([x,y]\) where \(x\in\graded_k\free_n\) and \(y\in\graded_{k_0-k}\free_n\) for some \(0<k<k_0\). If \(n\geqslant k_0\), then by the induction hypothesis,
\begin{align*}
x&=\sum_{f:\bld{k}\to\bld{n}}f_{\ast}(x_f)\\
y&=\sum_{g:\bld{k_0-k}\to\bld{n}}g_{\ast}(y_g)\\
\end{align*}
for some \(x_f\in\graded_k\free_k\) and \(y_g\in\graded_{k_0-k}\free_{k_0-k}\). For every pair \((f:\bld{k}\to\bld{n},g:\bld{k_0-k}\to\bld{n})\), form the pushout \(P_{f,g}\in\FI\) of the pullback of \(f\) and \(g\) to obtain the commutative diagram
\begin{equation*}
\begin{tikzcd}
& \bld{k} \ar[d,"i_{f,g}"] \ar[rdd,"f",bend left]\\
\bld{k_0-k} \ar[r,"j_{f,g}"'] \ar[rrd,"g"',bend right] & P_{f,g} \ar[rd,"f+g"']\\
&& \bld{n}
\end{tikzcd}
\end{equation*}
Clearly \(n(f,g)\coloneqq\vert P_{f,g}\vert\leqslant k_0\). We also have
\begin{align*}
[x,y]&=\sum_{\substack{f:\bld{k}\to\bld{n}\\g:\bld{k_0-k}\to\bld{n}}}[f_{\ast}(x_f),g_{\ast}(y_g)]\\
&=\sum_{\substack{f:\bld{k}\to\bld{n}\\g:\bld{k_0-k}\to\bld{n}}}(f+g)_{\ast}[(i_{f,g})_{\ast}(x_f),(j_{f,g})_{\ast}(y_g)],
\end{align*}
where each bracket in the latter sum is an element of \(\graded_{k_0}\free_{n(f,g)}\). Hence \(\graded_{k_0}\free\) is generated in degrees \(\leqslant k_0\).
\end{proof}

\paragraph{Free nilpotent groups}
Write \(N_n(k)\coloneqq\free_n/\gamma_{k+1}(\free_n)\) for \(k\geqslant1\). This is the free nilpotent group on \(n\) generators of nilpotency class \(k\). In particular, \(N_n(1)\cong\mathbb{Z}^n\); \(N_2(2)\) is the so-called Heisenberg group. We will denote the generators of \(N_n(k)\) coming from the generators \(x_1,\ldots,x_n\in\free_n\) by the same letters.

\(N_n(k)\) for a fixed \(k\) and varying \(n\) fit into the \(\FI\)-group \(N(k)\coloneqq\free/\gamma_{k+1}(\free)\). There are extensions of \(\FI\)-groups
\[0\to\graded_k\free\to N(k)\to N(k-1)\to1\]
for all \(k>1\). For every \(i<k\), there is a canonical isomorphism \(\gamma_i(\free)/\gamma_{k+1}(\free)\cong\gamma_i(N(k))\) induced by the quotient morphism \(\free\to N(k)\).

\paragraph{Automorphism groups}
Consider now the groups \(\Aut N_n(k)\) of automorphisms of \(N_n(k)\). For a fixed \(k\), these also form an \(\FI\)-group as follows: given \(f:\bld{n}\hookrightarrow\bld{m}\) and \(\phi\in\Aut N_n(k)\), \(f_{\ast}(\phi):N_m(k)\to N_m(k)\) is defined by
\[f_{\ast}(\phi)(x_i)=
\begin{cases}
x_i & \mbox{if } i\not\in\operatorname{im}f,\\
N(k)(f)\circ\phi(x_j) & \mbox{if } i=f(j), j\in\bld{n}.
\end{cases}\]
It is an easy check that \(f_{\ast}(\phi)\) is an automorphism (indeed, \(f_{\ast}(\phi^{-1})\) is its inverse) and that the assignment \(\phi\mapsto f_{\ast}(\phi):\Aut N_n(k)\to\Aut N_m(k)\) is a homomorphism.

There are morphisms of \(\FI\)-groups \(\rho(k):\Aut N(k)\to\Aut N(k-1)\) defined as follows: an automorphism \(\phi\in\Aut N_n(k)\) preserves \(\gamma_k(N_n(k))=\gamma_k(\free_n)/\gamma_{k+1}(\free_n)\) and therefore descends to an endomorphism \(\phi'\) of \(N_n(k-1)\). But \(\phi^{-1}\) descends to the inverse of \(\phi'\), so \(\phi'\in\Aut N_n(k-1)\). The assignment \(\phi\mapsto\phi'\) then defines the map \(\rho(k)\) of \(\FI\)-groups.

\paragraph{The kernel of \(\rho(k)\)} Suppose that \(\phi\in\Aut N_n(k)\) lies in the kernel of \(\rho_n(k)\). Then the set map \(N_n(k)\to\gamma_k(N_n(k))\cong\graded_k\free_n\) sending \(x\) to \(\phi(x)x^{-1}\) is in fact a homomorphism since \(\gamma_k(N_n(k))\) is central in \(N_n(k)\). This map descends to a homomorphism \(\phi':\mathbb{Z}^n\to\graded_k\free_n\), where \(\mathbb{Z}^n\) is the abelianisation of \(N_n(k)\). We obtain a homomorphism
\begin{align*}
\psi_n(k):\ker\rho_n(k)&\to\Hom(\mathbb{Z}^n,\graded_k\free_n)\\
\phi&\mapsto\phi'
\end{align*}
It is well known that \(\psi_n(k)\) is an isomorphism of abelian groups.

Let \(\mathbb{Z}^{\bullet}:\FI^{op}\to\Ab\) be the \(\FI^{op}\)-module taking \(\bld{n}\) to \(\mathbb{Z}^n\) and \(f:\bld{n}\hookrightarrow\bld{m}\) to the projection
\begin{align*}
f^{\ast}:\mathbb{Z}^m&\to\mathbb{Z}^n\\
x_i&\mapsto
\begin{cases}
0 & \mbox{if } i\notin\im f\\
x_j & \mbox{if } i=f(j)
\end{cases}
\end{align*}
Observe that the dual \(\FI\)-module \(\left(\mathbb{Z}^{\bullet}\right)^{\ast}\) is canonically isomorphic to \(\graded_1\free\). Let \(K(k)\) be the \(\FI\)-module given by
\[K(k):\FI\xrightarrow{\left(\mathbb{Z}^{\bullet}\right)^{op}\times\graded_k\free}\Ab^{op}\times\Ab\xrightarrow{\Hom(-,-)}\Ab.\]
We leave it to the reader to verify that the \(\psi_n(k)\) are components of a natural isomorphism \(\psi(k):\ker\rho(k)\to K(k)\).

\begin{lemma}\label{kerrho_lemma}
For every \(k\geqslant1\), \(\ker\rho(k)\) is a finitely generated \(\FI\)-module taking values in free abelian groups.
\end{lemma}
\begin{proof}
Since \(\ker\rho(k)\cong K(k)\), it is enough to prove the assertion for \(K(k)\).

Recall that there is a natural isomorphism \(\Hom(A,B)\cong A^{\ast}\otimes B\) for \(A\) a finitely generated free abelian group and \(B\in\Ab\), where \(A^{\ast}\) denotes the dual of \(A\). Hence there are isomorphisms \(K(k)\cong\left(\mathbb{Z}^{\bullet}\right)^{\ast}\otimes\graded_k\free\cong\graded_1\free\otimes\graded_k\free\). But \(\graded_1\free\) and \(\graded_k\free\) are finitely generated and take values in free abelian groups by Lemma~\ref{graded_pieces}. The result follows from Corollary~\ref{tensor_products_fg}.
\end{proof}

\paragraph{Johnson filtration}
There is also a homomorphism \(\Aut\free_n\to\Aut N_n(k)\) for every \(k\) and \(n\). Fix \(n\); then the kernels
\[\IA_n(k)\coloneqq\set{\phi\in\Aut\free_n}{\phi(x)x^{-1}\in\gamma_{k+1}(\free_n)\mbox{ for all }x\in\free_n}\]
of these homomorphisms form a filtration of \(\Aut\free_n\). We record the following theorem, originally due to Andreadakis \cite{and65}, for later use.

\begin{thm}[{\cite[Theorem 1.1]{and65}}]\label{free_johnson_central}
The filtration \(\lbrace\IA_n(k)\rbrace_{k\geqslant1}\) of \(\IA_n(1)\) is central.
\end{thm}

\section{Quotients of mapping class groups}\label{section:quotients of mcgs}
\paragraph{Mapping class groups}
Let \(\surf_{g,n}\) be an oriented surface of genus \(g\) with \(n\) boundary components. Fix collars \(\left[0,\epsilon\right)\times S^1\hookrightarrow\surf_{g,n}\) of the boundary components.

Let \(\Diff^+(\surf_{g,1},\partial)\) be the group of orientation-preserving diffeomorphisms of the surface which fix the collar pointwise. The mapping class group \(\Gamma_{g,1}\) of \(\surf_{g,1}\) is the group of isotopy classes of such diffeomorphisms:
\[\Gamma_{g,1}\coloneqq\Diff^+(\surf_{g,1},\partial)/\Diff_0^+(\surf_{g,1},\partial)\]
where \(\Diff_0^+(\surf_{g,1},\partial)\) is the path-component of the identity in \(\Diff^+(\surf_{g,1},\partial)\).

\begin{figure}
\centering
\def\svgwidth{15.9cm}
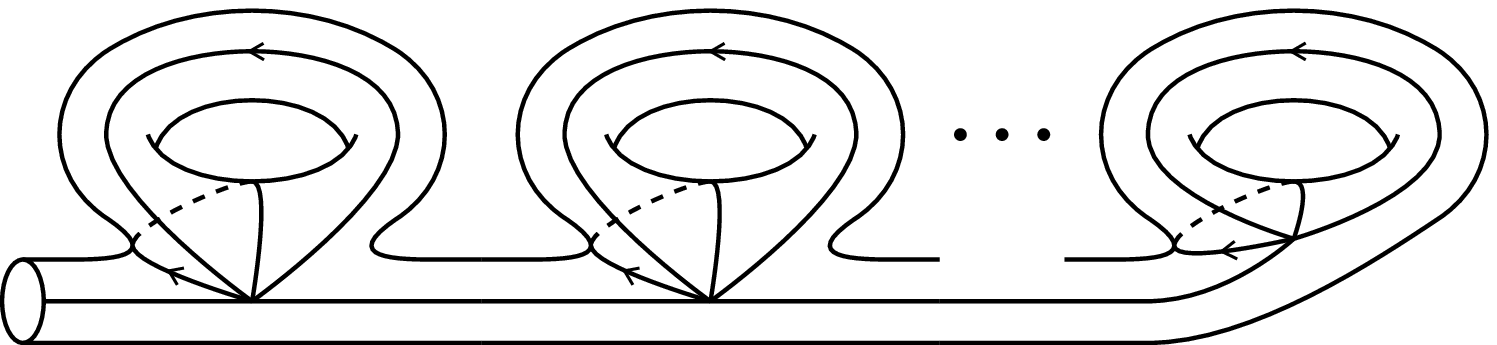
\caption{\(\surf_{g,1}\)}
\label{fig:surface}
\end{figure}

We fix a point \(\ast\in\partial\surf_{g,1}\), and we choose generators \(x_1,y_1,\ldots,x_g,y_g\) of \(\pi\coloneqq\pi_1(\surf_{g,1},\ast)\cong\free_{2g}\) as in Figure~\ref{fig:surface}. Then the boundary word \(\zeta\in\pi\) is equal to
\[\prod_i[x_i,y_i].\]
\(\pi\) admits an action of \(\Gamma_{g,1}\), yielding a homomorphism \(\Gamma_{g,1}\to\Aut\pi\). By the Dehn-Nielsen-Baer theorem, this homomorphism is injective with image precisely the stabiliser of \(\zeta\in\pi\):
\[\Gamma_{g,1}\cong\set{\phi\in\Aut\free_{2g}}{\phi(\zeta)=\zeta}.\]

\paragraph{Johnson filtration}
Just like in the case of \(\Aut\free_n\), we can filter \(\Gamma_{g,1}\) by the kernels \(\tor_{g,1}(k)\) of its action on the quotients \(\free_{2g}/\gamma_{k+1}(\free_{2g})\):
\[\tor_{g,1}(k)=\set{\phi\in\Gamma_{g,1}}{\phi(x)x^{-1}\in\gamma_{k+1}(\free_{2g})\mbox{ for all }x\in\free_{2g}}=\Gamma_{g,1}\cap\IA_{2g}(k).\]
This is the Johnson filtration \(\Gamma_{g,1}=\tor_{g,1}(0)\rhd\tor_{g,1}(1)\rhd\cdots\), defined by D. Johnson in the 80s to study the Torelli group \(\tor_{g,1}=\tor_{g,1}(1)\). We write
\[\Gamma_{g,1}(k)\coloneqq\Gamma_{g,1}/\tor_{g,1}(k)\]
for the quotient. \(\Gamma_{g,1}(k)\) is canonically isomorphic to the image of \(\Gamma_{g,1}\) in \(\Aut N_{2g}(k)\).

In the case \(k=1\), \(\pi/\gamma_2(\pi)\) is naturally isomorphic to \(H_g\coloneqq H_1(\surf_{g,1};\mathbb{Z})\cong\mathbb{Z}^{2g}\). The action of \(\Gamma_{g,1}\) on \(H_g\) preserves the algebraic intersection pairing, so the image of \(\Gamma_{g,1}\) in \(\Aut N_{2g}(1)\cong\GL(H_g)\) lies inside \(\Sp(H_g)\cong\SpZ{g}\). It is well known that this image is in fact the whole of \(\Sp(H_g)\).

\begin{lemma}\label{surf_johnson_central}
The filtration \(\{\tor_{g,1}(k)\}_{k\geqslant1}\) of \(\tor_{g,1}\) is central.
\end{lemma}
\begin{proof}
This is a classical result (it also follows immediately from Theorem~\ref{free_johnson_central}).
\end{proof}

\paragraph{\(\SI\)-modules from the Johnson filtration}
Suppose \(\phi:\surf_{g,1}\hookrightarrow\surf_{g',1}\) is an orientation-preserving embedding of surfaces. We can extend a given mapping class on \(\surf_{g,1}\cong\im\phi\) by the identity on the complement \(\surf_{g',1}\setminus\im\phi\) to obtain a mapping class on \(\surf_{g',1}\). This construction clearly produces a well-defined homomorphism \(\Gamma_{g,1}\to\Gamma_{g',1}\).
\begin{defn}\label{connecting_homs}
We write \(\phi_{\flat}:\Gamma_{g,1}\to\Gamma_{g',1}\) for the above homomorphism.
\end{defn}

When we picture surfaces \(\surf_{g,1}\) and \(\surf_{g',1}\) with \(g\leqslant g'\) as in Figure~\ref{fig:surface}, there is an obvious embedding \(T:\surf_{g,1}\hookrightarrow\surf_{g',1}\) as the subsurface consisting of the \(g\) rightmost tori. Algebraically, \(T_{\flat}:\Gamma_{g,1}\to\Gamma_{g',1}\) extends a given mapping class on \(\surf_{g,1}\) (thought of as an automorphism of \(\free_{2g}\)) to an automorphism of \(\free_{2g'}\) by the identity on the remaining generators \(x_i, y_i\), \(i>g\). It is easy to see that \(T_{\flat}\) preserves the Johnson filtration. In fact, this is true for any embedding:

\begin{lemma}\label{conn_homs_preserve_johnson}
Let \(\phi:\surf_{g,1}\hookrightarrow\surf_{g',1}\) be an embedding. Then \(\phi_{\flat}\) preserves the Johnson filtration.
\end{lemma}
\begin{proof}
First of all observe that \(\phi_{\flat}\) only depends on the isotopy class of \(\phi\), whence we may assume that \(\im\phi\) is disjoint from some collar of \(\partial\surf_{g',1}\) and therefore that \(\surf_{g',1}\setminus\im\phi\) is connected. Let \(S\) be the closure of \(\surf_{g',1}\setminus\im\phi\); by the classification of oriented surfaces, \(S\cong\surf_{g'-g,2}\). Hence we can extend \(\phi\) to a diffeomorphism \(\widetilde{\phi}\) of \(\surf_{g',1}\) such that \(\phi=\widetilde{\phi}\circ T\). But then \(\phi_{\flat}\) and \(T_{\flat}\) are conjugate by \([\widetilde{\phi}]\) in \(\Gamma_{g',1}\).
\end{proof}

Hence \(\phi_{\flat}\) descends to well-defined homomorphisms
\[\phi_{\flat}:\tor_{g,1}(k-1)/\tor_{g,1}(k)\to\tor_{g',1}(k-1)/\tor_{g',1}(k)\]
for each \(k>1\). We call these homomorphisms \(\phi_{\flat}\) as well, since there is no danger of confusion.

\begin{lemma}\label{embedding_lemma}\leavevmode
\begin{enumerate}[label={\normalfont (\roman*)}]
\item\label{realised_by_emb} Every symplectic map \(f:H_g\to H_{g'}\) is induced by an orientation-preserving embedding \(\surf_{g,1}\hookrightarrow\surf_{g',1}\) of surfaces.
\item If two orientation-preserving embeddings \(\phi,\phi':\surf_{g,1}\hookrightarrow\surf_{g',1}\) induce the same map on the first homology group, then there is some \(\psi\in\Diff^+(\surf_{g',1},\partial)\) which represents a class in \(\tor_{g',1}\) and such that \(\phi'\) and \(\psi\circ\phi\) are isotopic.
\end{enumerate}
\end{lemma}
\begin{proof}\leavevmode
\begin{enumerate}[label={\normalfont (\roman*)}]
\item There is certainly an embedding \(T:\surf_{g,1}\hookrightarrow\surf_{g',1}\) which induces the standard inclusion \(T_{\ast}:H_g\to H_{g'}\) that sends \([x_i]\) to \([x_i]\) and \([y_i]\) to \([y_i]\). But by Lemma~\ref{l:SI facts}, \(\Sp(H_{g'})\) acts transitively on the space of symplectic maps \(H_g\to H_{g'}\), so there is some \(s\in\Sp(H_{g'})\) such that \(s\circ T_{\ast}=f\). If \(\phi\in\Diff^+(\surf_{g',1},\partial)\) represents a lift of \(s\) to \(\Gamma_{g',1}\), then clearly \(\phi\circ T:\surf_{g,1}\hookrightarrow\surf_{g',1}\) is an embedding and \(\left(\phi\circ T\right)_{\ast}=f\).
\item This is the content of \cite[Lemma 4.1(ii)]{chpu15}.
\end{enumerate}
\end{proof}

Write \(Q_g(k)\coloneqq\tor_{g,1}(k-1)/\tor_{g,1}(k)\). We show that these groups fit into an \(\SI\)-module \(Q(k)\).

\begin{rmk}
Observe that the objects \(V_g\) of \(\SI\) are isomorphic (as symplectic spaces) to the \(H_g\). In the rest of this section, we assume the \(H_g\) \emph{are} the objects of \(\SI\), since it makes the construction of \(Q(k)\) more natural.
\end{rmk}

\begin{prop}
For every \(k>1\), there is an \(\SI\)-module \(Q(k)\) which takes \(H_g\) to \(Q_g(k)\) and \(f:H_g\to H_{g'}\) to \(\phi_{\flat}:Q_g(k)\to Q_{g'}(k)\), where \(\phi\) is as in Lemma~\ref{embedding_lemma}\ref{realised_by_emb}.\end{prop}
\begin{proof}
We define \(Q(k)\) as follows: \(Q(k)(H_g)=Q_g(k)\). Given a symplectic map \(f:H_g\to H_{g'}\), let \(\phi:\surf_{g,1}\hookrightarrow\surf_{g',1}\) be an embedding with \(\phi_{\ast}=f\). By Lemma~\ref{conn_homs_preserve_johnson}, the induced homomorphism \(\phi_{\flat}\) gives rise to a homomorphism \(Q_g(k)\to Q_{g'}(k)\). We claim that this induced map is independent of the choice of \(\phi\). Indeed, suppose that \(\phi':\surf_{g,1}\hookrightarrow\surf_{g',1}\) is another embedding with \(\phi'_{\ast}=f\). Let \(\psi\) be as in Lemma~\ref{embedding_lemma}. Then \(\phi'_{\flat}=\left(\psi\circ\phi\right)_{\flat}\), \emph{i.e.} \(\phi_{\flat}\) and \(\phi'_{\flat}\) are conjugate by \([\psi]\) in \(\Gamma_{g',1}\). But \([\psi]\in\tor_{g',1}\) and the Johnson filtration is central in \(\tor_{g',1}\), so \(\phi_{\flat}\) and \(\phi'_{\flat}\) induce the same homomorphism \(Q_g(k)\to Q_{g'}(k)\). Hence \(Q(k)\) is well defined. It is easy to check that \(Q(k)\) is a functor.
\end{proof}

Recall the functor \(X:\FI\to\SI\) from Definition~\ref{defn:SI-Mod to FI-Mod}: it sends \(\bld{n}\) to \(H_n\) and \(f:\bld{n}\to\bld{m}\) to the symplectic map
\begin{align*}
Xf:H_n&\to H_m\\
[x_i]&\mapsto[x_{f(i)}]\\
[y_i]&\mapsto[y_{f(i)}]
\end{align*}

\begin{thm}\label{Qk_finite_degree}
For every \(k>1\), \(X^{\ast}Q(k)\) is a finitely generated \(\FI\)-module. It takes values in free abelian groups.
\end{thm}
\begin{proof}
Write \(F\coloneqq X^{\ast}Q(k)\). For every \(n\), \(\tor_{n,1}(k)\) is the kernel of the action of \(\Gamma_{n,1}\) on \(N_{2n}(k)\), so there are injections \(j_n:\Gamma_{n,1}(k)\hookrightarrow\Aut N_{2n}(k)\). \(F_n=\tor_{n,1}(k-1)/\tor_{n,1}(k)\) is mapped into the kernel \(\ker\rho_{2n}(k)\) which is free abelian by Lemma~\ref{kerrho_lemma}. Hence \(F_n\) is also free abelian.

Moreover, we claim that the injections \(j_n:F_n\hookrightarrow\ker\rho_{2n}(k)\) are the components of a morphism \(j:F\hookrightarrow D^{\ast}\ker\rho(k)\) of \(\FI\)-modules, where \(D:\FI\to\FI\) is the duplication functor defined at the end of Section~\ref{subsection:FI modules}. Observe that we only need to check naturality with respect to the inclusions \(\bld{n-1}\hookrightarrow\bld{n}\) and automorphisms of \(\bld{n}\), as these arrows generate all the arrows in \(\FI\).

Naturality with respect to the inclusions \(\bld{n-1}\hookrightarrow\bld{n}\) follows easily from the definitions. Given a permutation \(\sigma\in\Aut_{\FI}(\bld{n})\), let \(\widetilde{\sigma}\in\Gamma_{n,1}\) be a lift of \(X\sigma\in\Aut_{\SI}(H_g)=\Sp(H_g)\). Let also \(s\in\Aut\free_{2n}\) be the obvious lift of \(X\sigma\), namely the automorphism sending \(x_i\mapsto x_{\sigma(i)}\) and \(y_i\mapsto y_{\sigma(i)}\). In general, \(s\) will not be in \(\Gamma_{n,1}\) because it does not preserve the boundary word \(\prod_i[x_i,y_i]\). We want to show the following diagram commutes
\begin{equation*}
\begin{tikzcd}
\tor_{n,1}(k-1)/\tor_{n,1}(k) \ar[r, hook, "j_n"] \ar[d, "\sigma_{\ast}"'] & \ker\rho_{2n}(k) \ar[d, "\sigma_{\ast}"]\\
\tor_{n,1}(k-1)/\tor_{n,1}(k) \ar[r, hook, "j_n"] & \ker\rho_{2n}(k)
\end{tikzcd}
\end{equation*}
where the left vertical arrow is conjugation by \(\widetilde{\sigma}\) while the right arrow is conjugation by \(s\).

Observe that \(s^{-1}\widetilde{\sigma}\in\IA_{2n}(1)\). Indeed, \(\widetilde{\sigma}(z_i)z_{\sigma(i)}^{-1}\in\gamma_{1}(\free_{2n})\), where \(z\) stands for either \(x\) or \(y\). Hence
\[s^{-1}\left(\widetilde{\sigma}(z_i)z_{\sigma(i)}^{-1}\right)=(s^{-1}\widetilde{\sigma})(z_i)z_i^{-1}\in\gamma_1(\free_{2n}),\]
so by Theorem~\ref{free_johnson_central}, \([s^{-1}\widetilde{\sigma},\phi]\in\IA_{2n}(k)\) for every \(\phi\in\tor_{n,1}(k-1)\). Therefore \(s\phi s^{-1}\) and \(\widetilde{\sigma}\phi\widetilde{\sigma}^{-1}\) induce the same automorphism of \(N_{2n}(k)\), which proves the claim.

But \(D^{\ast}\ker\rho(k)\) is finitely generated by Lemma~\ref{duplication_fg} since \(\ker\rho(k)\) is finitely generated. Hence \(F\) is finitely generated by Theorem~\ref{thm:FI noetherian}.
\end{proof}

\begin{rmk}
In particular, \(Q(k)\) has finite degree as an \(\SI\)-module by Lemma~\ref{l:degree of pullback}. We need, however, the stronger statement we made in the theorem, since it will allow us to argue that the exterior powers \(\Lambda^nQ(k)\) have finite degree as well.
\end{rmk}

\section{Homological stability for the groups \(\Gamma_{g,1}(k)\)}\label{section:proof of theorem}
\subsection{A general framework}
We will deduce Theorem~\ref{thm:hom stab} from Theorem~\ref{Qk_finite_degree} and homological stability for the groups \(\SpZ{n}\) using a bootstrapping argument.

\begin{defn}\leavevmode
\begin{enumerate}[label={\normalfont (\roman*)}]
\item A \emph{sequence of groups} is an \(\mathbb{N}\)-object in the category \(\Grp\), i.e. a composable sequence \(G_0\to G_1\to\cdots\) of group homomorphisms. A morphism of a sequence of groups is a natural transformation.
\item A \(G\)-\emph{module} \(M\) for a sequence of groups \(G\) is a composable sequence \(M_0\to M_1\to\cdots\) of abelian group homomorphisms, together with a \(G_n\)-action on \(M_n\) for each \(n\) such that the connecting homomorphisms \(M_{n-1}\to M_n\) are equivariant with respect to \(G_{n-1}\to G_n\).
\item A sequence of groups \(G\) has \emph{homological stability with coefficients} in a \(G\)-module \(M\) if, for every \(i\), the induced homomorphism
\[H_i(G_{n-1};M_{n-1})\to H_i(G_n;M_n)\]
is an isomorphism for all \(n\geqslant n(i)\) for some \(n(i)\geqslant1\).
\end{enumerate}
\end{defn}

\begin{rmk}
The category \(\Grp^{\mathbb{N}}\) of sequences of groups is complete, with limits computed pointwise, and the sequence \(1\to1\to\cdots\) is a null object. In particular, the category has all kernels. Also, a morphism \(f:G\to H\) is monic (resp. epic) if and only if every component \(f_n:G_n\to H_n\) is monic (resp. epic).
\end{rmk}

\begin{ex}[Sequences of groups]
The homomorphisms \(T_{\flat}:\Gamma_{g,1}\to\Gamma_{g',1}\) for \(g\leqslant g'\) defined in the previous section give rise to a sequence of groups
\[\Gamma\coloneqq\left(\Gamma_{0,1}\to\Gamma_{1,1}\to\cdots\right).\]
By Lemma~\ref{conn_homs_preserve_johnson}, \(\Gamma\) is filtered by sequences of groups arising from the Johnson filtration
\[\tor(k)\coloneqq\left(\tor_{0,1}(k)\to\tor_{1,1}(k)\to\cdots\right)\]
and the quotients also fit into sequences of groups
\[\Gamma(k)\coloneqq\left(\Gamma_{0,1}(k)\to\Gamma_{1,1}(k)\to\cdots\right)\]
for each \(k\geqslant1\).
\end{ex}
\begin{ex}[Modules of sequences of groups]\label{ex:seq_modules}
Consider the sequence \(\Gamma(1)\) which consists of the groups \(\Gamma_{n,1}(1)\cong\SpZ{n}\). The connecting homomorphisms \(\SpZ{(n-1)}\to\SpZ{n}\) take a symplectic \(\left(2n-2\right)\times\left(2n-2\right)\) matrix \(A\) to the symplectic \(2n\times2n\) matrix
\begin{equation*}
\left(
\begin{array}{c;{1pt/2pt}c}
A & 0\\
\hdashline[1pt/2pt]
0 & I_2
\end{array}
\right)
\end{equation*}
where \(I_2\) is the \(2\times2\) identity matrix. Another way to think about the connecting homomorphisms is in the context of the category \(\SI\): we have \(\SpZ{n}=\Aut_{\SI}(V_n)\), and the connecting homomorphism is induced by the monoidal product: \(\Aut_{\SI}(V_{n-1})\xrightarrow{(-)\boxtimes V_1}\Aut_{\SI}(V_n)\).

Hence an \(\SI\)-module \(F\) gives rise to a \(\Gamma(1)\)-module
\[F_0\xrightarrow{\left(i_0\right)_{\ast}}F_1\xrightarrow{\left(i_1\right)_{\ast}}\cdots\]
where \(i_n:V_n\hookrightarrow V_{n+1}\) is the inclusion `on the left', \emph{i.e.} the arrow \(V_n\boxtimes\left(V_0\to V_1\right)\). We will abuse notation and write \(F\) for this \(\Gamma(1)\)-module as well.
\end{ex}

\paragraph{Bootstrapping stability}
If \(f:G\to H\) is a morphism of sequences of groups and \(M\) is an \(H\)-module, we obtain a \(G\)-module \(f^{\ast}M\) by restricting the action of \(H_n\) on \(M_n\) along \(f_n:G_n\to H_n\). Suppose we have an extension
\[1\to G'\xrightarrow{\iota} G\to G''\to1\]
of sequences of groups and a \(G\)-module \(M\). Then for every \(i\), the sequence
\[\cdots\to H_i(G'_{n-1};\iota^{\ast}M_{n-1})\to H_i(G'_n;\iota^{\ast}M_n)\to\cdots\]
forms a \(G''\)-module \(H_i(G';\iota^{\ast}M)\), where the module structure comes from the conjugation action of \(G''_n\) on the homology of \(G'_n\). We can deduce from the Lyndon-Hochschild-Serre spectral sequence the following theorem.
\begin{thm}\label{bootstrapping}
Let \(1\to G'\xrightarrow{\iota} G\to G''\to1\) and \(M\) be as above. If for every \(i\), \(G''\) has homological stability with coefficients in \(H_i(G';\iota^{\ast}M)\), then \(G\) has homological stability with coefficients in \(M\).
\end{thm}
See \emph{e.g.} \cite{szy14} for a proof.

\subsection{Proof of the main result}
First of all recall the following well-known theorem which says that the groups \(\SpZ{n}\) enjoy homological stability with suitable twisted coefficients.

\begin{thm}\label{symp_stab}
If a module \(M\) of the sequence \(\Gamma(1)=\left(1\to\Sp(2,\mathbb{Z})\to\Sp(4,\mathbb{Z})\to\cdots\right)\) comes from an \(\SI\)-module of finite degree as in Example~\ref{ex:seq_modules}, then \(\Gamma(1)\) has homological stability with coefficients in \(M\).
\end{thm}

See \emph{e.g.} \cite{rww15} for a proof.

\begin{defn}\label{def:good_module}
We say a \(\Gamma(1)\)-module is \emph{good} if it satisfies the hypothesis of Theorem~\ref{symp_stab}.

We say a \(\Gamma(k)\)-module, \(k>1\), is \emph{good} if it is the restriction of a good \(\Gamma(1)\)-module along the quotient morphism \(\Gamma(k)\to\Gamma(1)\).
\end{defn}

Now we are ready to prove Theorem~\ref{thm:hom stab}.
\setcounter{letterctr}{0}
\begin{lettered_thm}
For all \(k\geqslant1\), \(\Gamma(k)\) enjoys homological stability with coefficients in every good \(\Gamma(k)\)-module.
\end{lettered_thm}
\begin{proof}
We proceed by induction on \(k\). The case \(k=1\) is precisely the content of Theorem~\ref{symp_stab}.

Suppose the claim has been established for \(\Gamma(k-1)\) for some \(k>1\), and let \(M\) be a good \(\Gamma(k)\)-module. There is an extension of sequences of groups
\[0\to Q(k)\to\Gamma(k)\to\Gamma(k-1)\to1\]
where \(Q(k)=\left(Q_0(k)\to Q_1(k)\to\cdots\right)\) is the obvious subobject of the sequence of groups \(\Gamma(k)\). Since \(Q(k)\) lies in the kernel of \(\Gamma(k)\to\Gamma(1)\) and \(M\) is pulled back along this map, the action of \(Q_n(k)\) on \(M_n\) is trivial for each \(n\). Each \(Q_n(k)\) is a finitely generated free abelian group by Theorem~\ref{Qk_finite_degree}, and so by the Universal Coefficient Theorem,
\[H_i(Q_n(k);M_n)\cong M_n\otimes_{\mathbb{Z}}\Lambda^iQ_n(k)\]
where \(\Gamma_{n,1}(k-1)\) acts diagonally.

Since \(M\) is good, it is induced from some \(\SI\)-module of finite degree which we also call \(M\). Then the \(\Gamma(k-1)\)-module \(H_i(Q(k);M)\) is induced from the \(\SI\)-module \(M\otimes\Lambda^iQ(k)\). Once we show that this \(\SI\)-module is good, we are done: by the induction hypothesis, \(\Gamma(k-1)\) then enjoys homological stability with coefficients in \(H_i(Q(k);M)\) and by Theorem~\ref{bootstrapping}, \(\Gamma(k)\) has homological stability with coefficients in \(M\).

By Lemma~\ref{l:degree of pullback}, it suffices to show that the pulled-back \(\FI\)-module \(X^{\ast}\left(M\otimes\Lambda^iQ(k)\right)\cong\left(X^{\ast}M\right)\otimes\Lambda^i\left(X^{\ast}Q(k)\right)\) has finite degree. \(X^{\ast}Q(k)\) is finitely generated by Theorem~\ref{Qk_finite_degree}, so \(\left(X^{\ast}Q(k)\right)^{\otimes i}\) is also finitely generated by Corollary~\ref{tensor_products_fg}. Quotients of finitely generated \(\FI\)-modules are clearly finitely generated, so \(\Lambda^i\left(X^{\ast}Q(k)\right)\) is finitely generated. Hence by Corollary~\ref{fin_deg_tensors}, \(\left(X^{\ast}M\right)\otimes\Lambda^i\left(X^{\ast}Q(k)\right)\) has finite degree.
\end{proof}

\subsection{Towards a computation of the stable homology}
An obvious question to ask now is: what is the stable homology of the \(\Gamma_{g,1}(k)=\Gamma_{g,1}/\mathcal{I}_{g,1}(k)\) for \(k>1\), at least rationally? In a recent preprint \cite{szy16}, Szymik shows that in the case of automorphism groups of free nilpotent groups (see Section~\ref{section:automs of free nilps}), the canonical homomorphisms \(\Aut N_n^k\to\GL(n,\mathbb{Z})\) are rational homology isomorphisms in the stable range. He does so by showing that for every \(k\), the homology Lyndon-Hochschild-Serre spectral sequence for the extension
\[1\to K_n^k\to\Aut N_n^k\to\GL(n,\mathbb{Z})\to1\]
where \(K_n^k\lhd\Aut N_n^k\) is the kernel, stably collapses on the second page, with all rows except the 0th stably vanishing. Hence the edge homomorphism \(H_{\ast}(\Aut N_n^k;\mathbb{Q})\to H_{\ast}(\GL(n,\mathbb{Z});\mathbb{Q})\) is stably an isomorphism.

It would be tempting to think that an analogous statement might hold in the case of the quotients \(\Gamma_{g,1}(k)\), but this is in fact not true. Indeed, consider the extensions
\begin{equation}\label{second_ext}
1\to Q_g(2)\to\Gamma_{g,1}(2)\to\SpZ{g}\to1.
\end{equation}
By a classical result of Johnson, \(Q_g(2)=\tor_{g,1}/\tor_{g,1}(2)\cong\Lambda^3H_g\). Hence
\[H^{\ast}(Q_g(2);\mathbb{Q})\cong\Lambda^{\ast}\left(\Lambda^3\widehat{H}_g^{\ast}\right)\]
where \(\widehat{H}_g\coloneqq H_g\otimes\mathbb{Q}\) is the standard rational representation of \(\Sp(2g,\mathbb{Q})\) and \(\widehat{H}_g^{\ast}\) denotes its \(\mathbb{Q}\)-linear dual. Hence the \(\SpZ{g}\)-action on \(H^{\ast}(Q_g(2);\mathbb{Q})\) comes from a natural \(\Sp(2g,\mathbb{Q})\)-action. Hence by the complete reducibility of finite-dimensional \(\Sp(2g,\mathbb{Q})\)-representations and by Borel's theorem \cite{bor81} on the vanishing of the cohomology of \(\SpZ{g}\) with coefficients in non-trivial irreducible \(\Sp(2g,\mathbb{Q})\)-representations,
\[H^p(\SpZ{g};H^q(Q_g(2);\mathbb{Q}))\cong H^p(\SpZ{g};\mathbb{Q})\otimes H^q(Q_g(2);\mathbb{Q})^{\SpZ{g}}\]
when \(g\) is large compared to \(p\).

The stable rational cohomology of \(\SpZ{g}\) is the polynomial algebra
\[\mathbb{Q}\left[c_1,c_3,c_5,\ldots\mid\deg(c_i)=2i\right].\]
Kawazumi and Morita \cite{kamo96} show that \(H^{\ast}(Q_g(2);\mathbb{Q})^{\SpZ{g}}\) is stably isomorphic to the polynomial algebra over \(\mathbb{Q}\) generated by the set \(\graph\) of finite non-empty connected trivalent graphs (with loops and multiple edges); the degree of a generator \(\Gamma\in\graph\) is the order of \(\Gamma\).

Let \(E_{\bullet}^{\ast\ast}\) be the Lyndon-Hochschild-Serre spectral sequence in rational cohomology associated to the extension \eqref{second_ext}. By the above, the second page is stably given by
\[E_2^{pq}\cong\mathbb{Q}[c_1,c_3,c_5,\ldots]\otimes\mathbb{Q}[\graph].\]
But this is concentrated in even bidegrees, so the spectral sequence stably collapses on the second page and we obtain the following result.
\begin{prop}
The stable rational cohomology ring of \(\Gamma_{g,1}(2)\) is isomorphic to
\[\mathbb{Q}[\graph\cup\{c_1,c_3,c_5,\ldots\}].\]
\end{prop}
Recall that the stable cohomology ring \(H^{\ast}(\Gamma_{\infty};\mathbb{Q})\) of the \(\Gamma_{g,1}\) is a polynomial algebra on the Miller-Morita-Mumford classes \(e_i\in H^{2i}(\Gamma_{\infty};\mathbb{Q})\). The classes \(c_i\) can be chosen so that they pull back to \(e_i\) along the projection \(\Gamma_{g,1}\to\SpZ{g}\). Also, Kawazumi and Morita \cite{kamo96} show that each \(\Gamma\in\graph\) (viewed as a stable cohomology class on \(\Gamma_{g,1}(2)\)) pulls back to \(\pm e_{\vert\Gamma\vert}\). Hence we see that the projections \(\Gamma_{g,1}\twoheadrightarrow\Gamma_{g,1}(2)\twoheadrightarrow\SpZ{g}\) induce these maps on stable cohomology:
\begin{align*}
\mathbb{Q}[c_1,c_3,c_5,\ldots] \hookrightarrow \mathbb{Q}[\graph\cup\{c_1,c_3,c_5,\ldots\}] & \twoheadrightarrow \mathbb{Q}[e_1,e_2,e_3,\ldots]\\
c_i & \mapsto e_i\\
\Gamma & \mapsto \pm e_{\vert\Gamma\vert}
\end{align*}

\section{Infinite loop-space structure}\label{section:infinite loop space}
There are `pair-of-pants products' \(\Gamma_{g,1}\times\Gamma_{h,1}\to\Gamma_{g+h,1}\) which in a sense generalise the connecting homomorphisms of Definition~\ref{connecting_homs}. They are constructed as follows: gluing \(\surf_{g,1}\) and \(\surf_{h,1}\) to two connected components of the pair of pants \(\surf_{0,3}\) yields a copy of \(\surf_{g+h,1}\). The product takes a mapping class on \(\surf_{g,1}\) and one on \(\surf_{h,1}\) and extends them by the identity on \(\surf_{0,3}\). These products make \(\coprod_{g\geqslant0}B\Gamma_{g,1}\) into a monoid, and we have
\[\Omega B\left(\coprod_{g\geqslant0}B\Gamma_{g,1}\right)\simeq\mathbb{Z}\times B\Gamma_{\infty}^+\]
by the group-completion theorem. Here \(+\) is the Quillen plus-construction and \(\Gamma_{\infty}\coloneqq\colim_g\Gamma_{g,1}\) is the stable mapping class group. Tillmann \cite{til00} proved that this is in fact an infinite loop-space.

We can likewise put a monoid structure on the space \(X_k\coloneqq:\coprod_{g\geqslant0}B\Gamma_{g,1}(k)\), and the obvious maps \(X_k\to X_{k'}\) for \(1\leqslant k'\leqslant k\leqslant\infty\) are monoid homomorphisms. In this section we prove the following result (where \(\Gamma_{g,1}(\infty)\coloneqq\Gamma_{g,1}\)).
\begin{lettered_thm}
For every \(k\geqslant1\), \(\Omega BX_k\simeq\mathbb{Z}\times B\Gamma_{\infty}(k)^+\) is an infinite loop-space. For every \(1\leqslant k'\leqslant k\leqslant\infty\), the quotient map \(X_k\to X_{k'}\) group-completes to a map of infinite loop-spaces.
\end{lettered_thm}

\subsection{The surface operad}
There is a coherence issue which we swept under the rug in the above definition of monoid structure on \(X_k\). If we pick the identifications \(\surf_{g,1}\cup_{\partial}\surf_{0,3}\cup_{\partial}\surf_{h,1}\cong\surf_{g+h,1}\) independently of one another, there is no reason why the resulting product structure on \(X_k\) should be associative. It is, however, possible to replace \(X_k\) with an equivalent space which is a strict monoid, and moreover an algebra over the surface operad \(\oper{M}\). By a theorem of Tillmann \cite{til00}, this monoid then group-completes to an infinite loop-space.

We briefly recall the construction of the operad \(\oper{M}\) from \cite{til00} as it will be useful in the proof of Theorem~\ref{thm:inf loops}. From now on, every surface \(\surf_{g,n+1}\) comes with a marked boundary component. We fix three `atomic' surfaces: a disc \(D\coloneqq\surf_{0,1}\), a pair of pants \(P\coloneqq\surf_{0,3}\) and a torus with two boundary components \(T\coloneqq\surf_{1,2}\).

Let \(\cat{E}_{g,n}\) be the groupoid with objects \((\surf,\sigma)\) where \(\surf\cong\surf_{g,n+1}\) is a surface which is built from the atomic surfaces by identifying the marked boundary on one surface with an unmarked boundary on another surface using the fixed parametrisation, and \(\sigma\) is an ordering of the free boundary components of \(\surf\). The set of morphisms \((\surf,\sigma)\to(\surf',\sigma')\) in \(\cat{E}_{g,n}\) is \(\pi_0\Diff(S,S';\partial)\), where \(\Diff(S,S';\partial)\) is the group of diffeomorphisms \(S\to S'\) which are compatible with the fixed collars, which take the marked boundary to the marked boundary and which preserve the ordering of the free boundary components. (Note that such diffeomorphisms must necessarily preserve orientation.) We will suppress \(\sigma\) from the notation.

Gluing of surfaces yields functors
\[\gamma:\cat{E}_{g,n}\times\cat{E}_{h_1,i_1}\times\cdots\times\cat{E}_{h_n,i_n}\to\cat{E}_{g+h,i}\]
where \(h=\sum_jh_j\) and \(i=\sum_ji_j\). These functors almost define the structure of an operad on the sequence of spaces \(\{\coprod_gB\cat{E}_{g,n}\}_{n\geqslant0}\) --- they are associative and equivariant with respect to the appropriate actions of symmetric groups, but there is no unit in \(\coprod_gB\cat{E}_{g,1}\). This was remedied by Tillmann \cite{til00} (with a minor mistake corrected by Wahl \cite{wah04}):
\begin{thm}
There are full subgroupoids \(\cat{S}_{g,n}\hookrightarrow\cat{E}_{g,n}\), retractions \(R:\cat{E}_{g,n}\to\cat{S}_{g,n}\) and functors \(\overline{\gamma}\) which make the diagrams
\[
\begin{tikzcd}
\cat{E}_{g,n}\times\cat{E}_{h_1,i_1}\times\cdots\times\cat{E}_{h_n,i_n} \ar[r, "\gamma"] \ar[d, "R"'] & \cat{E}_{g+h,i} \ar[d, "R"]\\
\cat{S}_{g,n}\times\cat{S}_{h_1,i_1}\times\cdots\times\cat{S}_{h_n,i_n} \ar[r, "\overline{\gamma}"] & \cat{S}_{g+h,i}
\end{tikzcd}
\]
commute and which make the sequence \(\{\coprod_gB\cat{S}_{g,n}\}_{n\geqslant0}\) of spaces into an operad. Moreover, this data can be chosen in such a way that the pair-of-pants product \(\overline{\gamma}(P;-,-)\) is associative in this operad.
\end{thm}

Note that provided the functors \(\overline{\gamma}\) exist, they are uniquely specified by the commutativity condition alone.

\begin{defn}
The operad \(\{\coprod_gB\cat{S}_{g,n}\}_{n\geqslant0}\) of the above theorem is the surface operad \(\oper{M}\).
\end{defn}

The pair-of-pants product makes \(\oper{M}\)-algebras into strict monoids. The key result of \cite{til00} can now be stated as follows.

\begin{thm}\label{grp_comp_M_alg}
Group-completion is a functor from \(\cat{M}\)-algebras to infinite loop-spaces.
\end{thm}

\subsection{The \(\oper{M}\)-algebras \(B\cat{S}_0(k)\)}
We construct a sequence of \(\oper{M}\)-algebras \(B\cat{S}_0(k)\) such that each \(B\cat{S}_0(k)\) is homotopy-equivalent to \(\coprod_gB\Gamma_{g,1}(k)\).

Fixing a base-point on \(S^1\) yields, via the fixed parametrisations, a choice of base-point on every boundary component of every object of \(\cat{E}_n=\coprod_g\cat{E}_{g,n}\). We define the base-point of \(\surf\in\cat{E}_n\) to be the base-point on the marked boundary, and \(\pi_1\surf\) will always be understood as the fundamental group at this base-point.

We now restrict our attention to the groupoid \(\cat{E}_0\) of surfaces with a unique boundary component. The elements of \(\Diff(\surf,\surf';\partial)\) preserve base-points for all \(\surf,\surf'\in\cat{E}_0\), so we get a well-defined set map
\[\cat{E}_0(\surf,\surf')\to\Iso(\pi_1(\surf),\pi_1(\surf')).\]
This map is in fact injective by the Dehn-Nielsen-Baer theorem. This motivates the following definition.

\begin{defn}
For each \(k\geqslant1\), we define the equivalence relation \(\sim_k\) on the hom-sets of \(\cat{E}_0\) as follows: \(\phi,\phi':\surf\to\surf'\) have \(\phi\sim_k\phi'\) if and only if \(\phi\) and \(\phi'\) induce the same isomorphism
\[\pi_1\surf/\gamma_k(\pi_1\surf)\to\pi_1\surf'/\gamma_k(\pi_1\surf').\]
It can be seen that these relations respect composition in \(\cat{E}_0\). Hence they give rise to quotient groupoids \(\cat{E}_0(k)\coloneqq\cat{E}_0/\sim_k\) and \(\cat{S}_0(k)\coloneqq\cat{S}_0/\sim_k\).
\end{defn}

Observe that for every \(k\) and \(g\), \(\cat{S}_{g,0}/\sim_k\) (and \(\cat{E}_{g,0}/\sim_k\)) is equivalent to \(\Gamma_{g,1}(k)\), so we have
\[B\cat{S}_0(k)\simeq\coprod_gB\Gamma_{g,1}(k).\]
We claim that \(B\cat{S}_0(k)\) is an \(\oper{M}\)-algebra. First we need the following proposition which relates the functors \(\gamma\) and \(\sim_k\).

\begin{prop}
Suppose that \(\phi:A\to A'\) is a morphism in \(\cat{E}_n\) and \(\psi_i\sim_k\psi'_i:B_i\to B'_i\) are pairs of \(\sim_k\)-equivalent morphisms in \(\cat{E}_0\) for \(1\leqslant i\leqslant n\). Then
\[\gamma(\phi;\psi_1,\ldots,\psi_n)\sim_k\gamma(\phi;\psi'_1,\ldots,\psi'_n).\]
\end{prop}
\begin{proof}
Let us first assume that \(n=1\), \(\phi=\id_A\) and \(B_1=B'_1\). We drop the subscripts from \(B\) and \(\psi,\psi'\).

By the van Kampen theorem, \(\pi_1(\gamma(A;B))\cong\pi_1(A)\ast_{\partial}\pi\), where \(\pi\) is a subgroup of \(\pi_1(\gamma(A;B))\) isomorphic to \(\pi_1(B)\) and \(\ast_{\partial}\) denotes amalgamation over the identified boundaries of \(A\) and \(B\), and similarly \(\pi_1(\gamma(A;B'))\cong\pi_1(A)\ast_{\partial}\pi'\). The isomorphisms \(\pi\cong\pi_1 B\) and \(\pi'\cong\pi_1 B'\) are not canonical, but instead depend on a choice of paths in \(A\) from the base-point of \(A\) to the base-point on its free boundary. If we let them both be induced by the same path, we obtain the diagram
\begin{equation*}
\begin{tikzcd}
\pi_1(A)\ast_{\partial}\pi_1(B)
\ar[r,"\sim"]
\ar[d,"\id\ast_{\partial}\psi_{\ast}"'] &
\pi_1(\gamma(A;B))
\ar[d,"\gamma(A;\psi)_{\ast}"]\\
\pi_1(A)\ast_{\partial}\pi_1(B')
\ar[r,"\sim"] &
\pi_1(\gamma(A;B'))
\end{tikzcd}
\end{equation*}
which commutes since the class \(\gamma(A;\psi)\) of diffeomorphisms \(\gamma(A;B)\to\gamma(A;B')\) can be represented by a diffeomorphism which restricts to the identity on \(A\).

We have an analogous diagram where \(\psi\) is replaced with \(\psi'\). Upon applying the functor \(G\mapsto G/\gamma_k(G):\Grp\to\Grp\) to both diagrams, the left vertical arrows become equal since \(\psi\sim_k\psi'\). The horizontal arrows are also equal (they are equal even before we quotient out by \(\gamma_k\)), and so \(\gamma(A;\psi)\sim_k\gamma(A;\psi')\).

Now suppose \(B\neq B'\). If \(\psi,\psi':B\to B'\) are \(\sim_k\)-equivalent, then so are \(\id,\psi^{-1}\psi':B\to B\). Hence by the above, \(\gamma(A;\id)\sim_k\gamma(A;\psi^{-1}\psi')\). But we clearly have \(\gamma(A;\id)=\id_{\gamma(A;B)}\) and \(\gamma(A;\psi^{-1}\psi')=\gamma(A;\psi)^{-1}\gamma(A;\psi')\), so \(\gamma(A;\psi)\sim_k\gamma(A;\psi')\).

For the most general case, note that
\[\gamma(\phi;\psi_1,\ldots,\psi_n)=\gamma(\phi;B'_1,\ldots,B'_n)\prod_{i=1}^n\gamma(A;B_1,\ldots,B_{i-1},\psi_i,B'_{i+1},\ldots,B'_n)\]
and
\[\gamma(A;B_1,\ldots,B_{i-1},\psi_i,B'_{i+1},\ldots,B'_n)=\gamma(A_i;\psi_i)\]
where \(A_i\coloneqq\gamma(A;B_1,\ldots,B_{i-1},-,B'_{i+1},\ldots,B'_n)\). By the above, \(\gamma(A_i,\psi_i)\sim_k\gamma(A_i;\psi'_i)\), finishing the proof.
\end{proof}

Now we are ready to prove Theorem~\ref{thm:inf loops}; it follows easily using Theorem~\ref{grp_comp_M_alg} from the following result.

\begin{prop}
For every \(k\), the structure functors \(\overline{\gamma}:\cat{S}_n\times\cat{S}_0^n\to\cat{S}_0\) descend to functors
\[\theta:\cat{S}_n\times\cat{S}_0(k)^n\to\cat{S}_0(k)\]
along the projection \(\cat{S}_0\to\cat{S}_0(k)\). This induces an \(\oper{M}\)-algebra structure on each \(B\cat{S}_0(k)\). For every \(1\leqslant k'\leqslant k\leqslant\infty\), the quotient functor \(\cat{S}_0(k)\to\cat{S}_0(k')\) gives rise to a map of \(\oper{M}\)-algebras upon realisation.
\end{prop}
\begin{proof}
Since the projection functor \(\cat{S}_0\to\cat{S}_0(k)\) is full (and the identity on objects), for every \(n\) there is at most one \(\theta\) making the right square in the diagram
\begin{equation}\label{extheta}
\begin{tikzcd}
\cat{E}_n\times\cat{E}_0^n \ar[r,"R"] \ar[d,"\gamma"'] & \cat{S}_n\times\cat{S}_0^n \ar[r] \ar[d,"\overline{\gamma}"'] & \cat{S}_n\times\cat{S}_0(k)^n \ar[d,dashed,"\theta"]\\
\cat{E}_0 \ar[r,"R"'] & \cat{S}_0 \ar[r] & \cat{S}_0(k)
\end{tikzcd}
\end{equation}
commute. It is easily seen that if \(\theta\) exist for all \(n\), they satisfy the axioms for an algebra over an operad, and any quotient map \(\cat{S}_0(k)\to\cat{S}_0(k')\) commutes with the \(\theta\), giving rise to an \(\oper{M}\)-algebra map \(B\cat{S}_0(k)\to B\cat{S}_0(k')\). So we only need to check existence.

\(R\), being a retraction of the inclusion of a full subgroupoid, is fully faithful, so for every \(\surf\in\cat{E}_{g,0}\), there is a unique \(\Phi_{\surf}:\surf\to R\surf\) which is mapped to the identity on \(R\surf\) by \(R\). But then every morphism \(\phi:\surf\to\surf'\) in \(\cat{E}_{g,0}\) has
\[R\phi=\Phi_{\surf'}\circ\phi\circ\Phi_{\surf}^{-1}.\]
Hence for a parallel pair of morphisms \(\phi,\phi':\surf\to\surf'\), \(\phi\sim_k\phi'\) if and only if \(R\phi\sim_kR\phi'\).

To show that \(\theta\) is well defined, let \(\psi:A\to A'\) be an arrow in \(\cat{S}_n\) and let \(\phi_i\sim_k\phi'_i:B_i\to B'_i\) be \(n\) pairs of \(\sim_k\)-equivalent morphisms in \(\cat{S}_0\). Then \(\gamma(\psi;\phi_1,\ldots,\phi_n)\sim_k\gamma(\psi;\phi'_1,\ldots,\phi'_n)\) in \(\cat{E}_0\) by the preceding proposition. Hence by the above observation about \(R\) and using the fact that on \(\cat{S}_n\times\cat{S}_0^n\), \(\overline{\gamma}=R\circ\gamma\), we see that
\[\overline{\gamma}(\psi;\phi_1,\ldots,\phi_n)\sim_k\overline{\gamma}(\psi;\phi'_1,\ldots,\phi'_n),\]
and so \(\overline{\gamma}\) descends to a well-defined functor \(\theta\).
\end{proof}

\bibliographystyle{abbrvnat}
\bibliography{article2}
\end{document}